\documentclass[preprint,12pt]{elsarticle}
\usepackage[utf8]{inputenc}

\usepackage{fullpage}

\usepackage{amsmath}
\usepackage{amssymb}
\usepackage{latexsym}
\usepackage{amsthm}
\usepackage[retainorgcmds]{IEEEtrantools}

\newtheorem{thm}{Theorem}
\newtheorem{lem}[thm]{Lemma}

\newtheorem{clm}{Claim}
\newtheorem*{clm*}{Claim}

\usepackage{minitoc}

\usepackage{ifpdf}

\journal{Discrete Mathematics}

\begin{document}

\begin{frontmatter}



\title{Point determining digraphs, $\{0,1\}$-matrix partitions, and dualities in full homomorphisms}


\author{Pavol Hell}
\author{C\'esar Hern\'andez-Cruz}
\address{School of Computing Science\\
Simon Fraser University\\
Burnaby, B.C., Canada V5A 1S6}

\begin{abstract}
We prove that every point-determining digraph $D$ contains a vertex $v$ such that $D-v$ is also point determining.   We apply this result to show that for any $\{0,1\}$-matrix $M$, with $k$ diagonal zeros and $\ell$ diagonal ones, the size of a minimal $M$-obstruction is at most $(k+1)(\ell+1)$.   This extends the results of Sumner, and of Feder and Hell, from undirected graphs and symmetric matrices to digraphs and general matrices.
\end{abstract}

\begin{keyword}


\end{keyword}

\end{frontmatter}


\section{Introduction}

We consider partitions of a digraph $D$ into sets that satisfy certain internal constraints (the set induces an independent set or a clique), and external constraints (a sets is completely adjacent or completely non-adjacent to another set). These constraints are encoded in a $\{0,1\}$-matrix $M$ defined below. We assume that the digraph $D$ has no loops. (We will allow loops, but only in a digraph that will be denoted exclusively by $H$.) A {\em strong clique} of $D$ is a set $C$ of vertices such that for any two distinct vertices $x, y \in C$ both arcs $(x, y), (y, x)$ are in $D$; and an {\em independent set} of $D$ is a set $I$ of vertices such that for any two vertices $x, y \in C$ neither pair $(x, y), (y, x)$ is an arc of $D$. Let $S, S'$ be two disjoint sets of vertices of $D$: we say that $S$ is {\em completely adjacent to} $S'$ (or $S'$ is completely adjacent from $S$) if for any $x \in S, x' \in S'$, the arc $(x, x')$ is in $D$; and we say that $S$ is {\em completely non-adjacent to} $S'$ (or $S'$ is completely non-adjacent from $S$) if for any $x \in S, x' \in S'$, the pair $(x, x')$ is not an arc of $D$.

Throughout this paper, $M$ will be a $\{0,1\}$-matrix with $k$ diagonal $0$'s and $\ell$ diagonal $1$'s. For convenience we shall assume that the rows and columns of $M$ are ordered so that the first $k$ diagonal entries are $0$, and the last $\ell$ diagonal entries are $1$. (Thus $k+\ell$ is the size of the matrix.)

An $M$-{\em partition} of a digraph $D$ is a partition of its vertex set $V(D)$ into parts $V_1, V_2, \dots, V_{k+\ell}$ such that 

\begin{itemize}
\item
$V_i$ is an independent set of $D$ if $M(i,i)=0$
\item
$V_i$ is a strong clique of $D$ if $M(i,i)=1$
\item
$V_i$ is completely non-adjacent to $V_j$ if $M(i,j)=0$
\item
$V_i$ is completely adjacent to $V_j$ if $M(i,j)=1$
\end{itemize}

In \cite{motwa} we introduced a more general version of matrix partitions, in which matrices are allowed to have an $*$ entry implying no restriction on the corresponding set, or pair of sets. For a survey of results on $M$-partitions we direct the reader to \cite{survey}.

A {\em full homomorphism} of a digraph $D$ to a digraph $H$ is a mapping $f : V(D) \to V(H)$ such that for distinct vertices $x$ and $y$, the pair $(x,y)$ is an arc of $D$ if and only if $(f(x),f(y))$ is an arc of $H$. The following observation is obvious: let $H$ denote the digraph whose adjacency matrix is $M$. (Note that $H$ has loops if $\ell > 0$.) Then $D$ admits an $M$-partition if and only if it admits a full homomorphism to $H$. It should be pointed out that our definition of full homomorphism (in particular the requirement that $x, y$ be distinct) is taylored to correspond to matrix partitions as defined in \cite{motwa}. The standard definition \cite{hombook,ales,ball} does not require this distinctness; this accounts for small discrepancies between the results of this paper and that of \cite{ball}. However, when $H$ has no loops, i.e., when $\ell=0$, the two definitions coincide.

Undirected graphs are viewed in this paper as special cases of digraphs, i.e., each undirected edge $xy$ is viewed as the two arcs $(x,y), (y,x)$. For a symmetric $\{0,1\}$-matrix $M$, the same definition applies to define an $M$-partition of an undirected graph $G$ \cite{motwa,survey}.

The questions investigated here have been studied for undirected graphs in \cite{fedhel,ball}, cf. \cite{survey}. It is shown in \cite{fedhel,ball} that for any symmetric $\{0,1\}$-matrix $M$ (i.e., any undirected graph $H$ with possible loops) there is a finite set $\cal G$ of graphs such that $G$ admits an $M$-partition (i.e., a full homomorphism to $H$) if and only if it does not contain an induced subgraph isomorphic to a member of $\cal G$. This property is what \cite{ball} calls a {\em duality} of full homomorphisms. Alternately \cite{survey}, we define a {\em minimal obstruction} to $M$-partition to be a digraph $D$ which does not admit an $M$-partition, but such that for any vertex $v$ of $D$, the digraph $D - v$ does admit an $M$-partition. Thus the results of \cite{ball,fedhel} imply that each symmetric $\{0,1\}$-matrix $M$ has only finitely many minimal graph obstructions. In \cite{fedhel} it is shown that these minimal graph obstructions have at most $(k+1)(\ell+1)$ vertices, and that there are at most two minimal graph obstructions with precisely $(k+1)(\ell+1)$ vertices. For the purposes of this proof, the authors of \cite{fedhel} consider the following concept. A graph is {\em point determining} if distinct vertices have distinct open neighbourhoods. According to Sumner \cite{sumner}, each point determining graph $H$ contains a vertex $v$ such that $H - v$ is also point determining; the authors of \cite{fedhel} derived their bound by proving a refined version of Sumner's result.

For digraphs (and $\{0,1\}$-matrices $M$ that are not necessarily symmetric), it is still true that each $M$ has at most a finite set of minimal digraph obstructions \cite{ball,survey}. In this paper we prove that the optimal bound still applies, i.e., that it is still the case that each minimal digraph obstruction has at most $(k+1)(\ell+1)$ vertices. (This was conjectured in earlier versions of \cite{survey}.) For this purpose we define a digraph version of point determination and prove the analogue of Sumner's result. Since undirected graphs can be viewed as symmetric digraphs, our results imply the $(k+1)(\ell+1)$ bound for graphs from \cite{fedhel}, as well as the basic version of Sumner's result.

We leave open the question whether a $\{0,1\}$-matrix $M$ always has at most two minimal digraph obstructions with $(k+1)(\ell+1)$ vertices; we did not find a counterexample.

In Section 2, we prove the above digraph version of Sumner's theorem, using the tools from \cite{fedhel}. In Section 3 we use this result to derive our $(k+1)(\ell+1)$ bound for the size of a minimal $M$-obstruction which has no (true or false) twins. In Section 4 we do the same for minimal $M$-obstructions that do have twins.


\section{Point-determining digraphs}

Let $D$ be a digraph and let $u,v,w$ be distinct vertices in $D$; we say that vertex $w$  {\em distinguishes} vertices $u,v$ in $D$ if exactly one of $u,v$ is in the in-neighborhood of $w$, or exactly one of $u,v$ is in the out-neighborhood of $w$.  We say that $u, v$ are {\em twins} in $D$ if there is no vertex that distinguishes them in $D$. We say that twins $u,v$ are {\em true} twins if $\{ u, v \}$ is a strong clique and {\em false} twins if $\{u,v\}$ is an independent set. We say that a digraph is {\em point-determining} if it does not contain a pair of false twins. Note that $D$ has no true twins if and only if the complement of $D$ is point-determining. 

In this section we will prove the following digraph analog to Sumner's Theorem.

\begin{thm} \label{nott}
If $D$ is a point-determining digraph, then there exists at least one vertex $v \in V(D)$ such that $D - v$  is point-determining.
\end{thm}

To prove this we will consider the notion of a triple in a point-determining digraph (cf. \cite{fedhel} for an analogous undirected concept). Let $D$ be a point-determining digraph. A triple $T = (x, \{y, z\})$ of G consists of a vertex $x$ of $D$, called the red vertex of $T$, and an unordered pair $\{y, z\}$ of vertices of $D$, called the green vertices of $T$, such that $y, z$ are false twins in $D-x$. (Thus $x$ is the only vertex of G that distinguishes $y$ and $z$.)   We begin with two lemmas.

\begin{lem} \label{triples}
Let $D$ be a point-determining digraph, and let $T_1$ and $T_2$ be two triples of $D$.   If $T_1$ and $T_2$ intersect in a vertex that is green in $T_1$ and red in $T_2$, then they intersect in another vertex that is green in $T_2$ and red in $T_1$.
\end{lem}

\begin{proof}
Consider two triples that share a vertex $z$ which is red in one triple and green in the other, say triples $T_1=(z, {u, v})$ and $T_2 = (x, {y, z})$. If $\{ x, y \} \cap \{u, v\} = \varnothing$, then since $z$ is the unique vertex distinguishing $u$ and $v$, the vertex $y$ does not distinguish $u$ and $v$. This means that one of the vertices $u$, $v$ distinguishes $y$ and $z$, which contradicts the fact that $(x, \{y, z\})$ is a triple of $D$ (\emph{i.e.}, $x$ is the only vertex of $D$ distinguishing $y$ and $z$). If $y \in \{u, v\}$ and $x \notin \{u, v \}$, say, $y = u$ and $v  \ne x$, then 
$v$ is not adjacent to $u = y$, so $v$ is not adjacent to $z$, because $(x, \{y, z\})$ is a triple and $v \ne x$. The vertices $u = y$ and $z$ are not  adjacent either, as $(x, \{y, z\})$ is a triple; this contradicts the fact that $(z, \{u,
v\})$ is a triple. Therefore $x$ must be one of $u, v$.
\end{proof}

\begin{lem} \label{triplesnored}
Let $D$ be a point-determining digraph.   There exists at least one vertex in $D$ that is red in no triple of $D$.
\end{lem}

Lemma \ref{triples} will be used implicitly a considerable number of times in the following proof.

\begin{proof}
If there are no triples, we are done. Else take a triple $(v_1, \{u_1,u_2\})$. If $u_1$ or $u_2$ is not red in any triple, we are done. Else there are triples $(u_1, \{ v_1, v_2'\})$ and $(u_2,\{v_1,v_2\})$ such that $v_2' \ne u_2$ or $v_2 \ne u_1$, otherwise $\{ v_1, u_1, u_2 \}$ would be an independent set, contradicting that $(v_1, \{u_1,u_2\})$ is a triple.   We will assume without loss of generality that $(u_2, \{v_1, v_2\})$ is a triple with $v_2 \ne u_1$.

If $v_2$ is not red in any triple, we are done.   Else, there exists a triple $(v_2, \{u_2,u_3\})$.   It follows from Lemma \ref{triples} that $u_3 \ne v_1$, and $u_3 \ne u_1$ because otherwise $v_1$ and $v_2$ would distinguish vertices $u_1$ and $u_2$, a contradiction.   If $u_3$ is not red in any triple, we are done.   Else there exists a triple $(u_3, \{ v_2, v_3 \})$.   Again, considering the triple $(u_2, \{ v_1, v_2 \})$ Lemma \ref{triples} implies that $u_3 \ne u_2$.   Also, $v_3 \ne v_1$, otherwise $u_2$ and $u_3$ would distinguish $v_1$ and $v_2$, a contradiction.   If $v_3 = u_1$, then $u_1$ and $v_2$ are non-adjacent by the definition of triple.   Since $\{ v_1, v_2 \}$ is an independent set and $u_3$ is the only vertex distinguishing $u_1$ and $v_2$, then $\{v_1, u_1 \}$ is also an independent set.   Recall that $u_1$ is red in some triple $(u_1, \{v_1, v_2'\})$, hence it is clear that $v_2' \ne u_2$, otherwise $\{ v_1, u_1, u_2 \}$ would be an independent set.   
 Lemma \ref{triples} implies that $v_1 = u_3$ or $v_2' = u_3$.   We have already proved that $u_3 \ne v_1$, so $v_2' = u_3$ and $\{ v_1, u_3 \}$ is an independent set, but $v_2$ is the only vertex distinguishing $u_2$ and $u_3$, so $v_1$ and $u_2$ are non-adjacent and $\{v_1, u_1, u_2\}$ is an independent set, contradicting the definition of triple, therefore $v_3 \ne u_1$.

Again, if $v_3$ is not red in any triple, we are done.   Else, it is clear that we can continue this argument until we find a vertex that is not red in any triple or until we repeat some vertex.   Let us assume that there is no vertex that is not red in any triple, then we have two finite sequences of triples: $\big\{ (v_i, \{u_i, u_{i+1}\}) \big\}$ and $\big\{ (u_{i+1}, \{v_i, v_{i+1}\}) \big\}$.   An argument similar to the previous one shows that a repeated vertex cannot be equal to one of the previous five vertices in the sequence $(u_1, v_1, u_2, v_2, \dots)$.   Also, it is easy to observe using an inductive argument and the definition of triple that the sets $\{ u_1, u_2,  \dots\}$ and $\{v_1, v_2, \dots \}$ are independent in $D$.   We will assume without loss of generality that the first repeated vertex is either $u_{n+1}=u_1$, $u_{n+1} = v_1$, $v_n = u_1$ or $v_n = v_1$.

If the first repeated vertex is $u_{n+1} = u_1$ then, since $v_2$ is the only vertex distinguishing $u_2$ and $u_3$, the adjacency type between $v_1$ and $u_3$ must be the same as the adjacency type between $v_1$ and $u_2$.   An inductive argument shows that $u_2$ has the same adjacency type with respect to every vertex $u_i$ with $2 \le i \le n+1$.   But $u_{n+1} = u_1$, contradicting that $v_1$ distinguishes $u_1$ and $u_2$.

If the first repeated vertex is $u_{n+1} = v_1$ then, Lemma \ref{triples} implies that $v_n = u_1$ or $v_n = u_2$, contradicting that the first repeated vertex is $u_{n+1}$.

If the first repeated vertex is $v_n = v_1$ then, Lemma \ref{triples} implies that $u_n = u_1$ or $u_n = u_2$, contradicting that the first repeated vertex is $v_n$.

For the final case, the first repeated vertex is $v_n = u_1$.   Since $u_1$ is a green vertex of the triples $(v_1, \{ u_1, u_2 \})$ and  $(u_n, \{v_{n-1}, u_1 \})$ and a red vertex in some triple $(u_1, \{ x_1, x_2 \})$, the only possibilities for $x_1$ and $x_2$ according to Lemma \ref{triples} are: $u_n = x_1 = v_1$, or $u_n = x_2 = v_1$, or $x_1 = v_1$ and $x_2 = u_n$, or $x_1 = u_n$ and $x_2 = v_1$.   In the first two cases the choice of $v_n$ as the first repeated vertex would be contradicted.   The third and fourth cases are equivalent, so we will consider only the case $x_1 = v_1$ and $x_2 = u_n$.   Recall that $\{u_1, \dots, u_n\}$ and $\{v_1, \dots, v_n\}$ are independent sets.   In particular $\{ u_1, u_n \}$ and $\{ v_1, v_n \}$ are independent sets, contradicting that $u_1 = v_n$ distinguishes $u_n$ and $v_1$.

Since in every case we reach a contradiction, there must be at least one vertex that is not red in any triple of $D$.
\end{proof}

The proof of Theorem \ref{nott} is now an immediate consequence of our lemmas.

\begin{proof}[Proof of Theorem \ref{nott}]
By the definition of triple, it suffices to choose a vertex that is not red in any triple of $D$ from Lemma \ref{triplesnored}.
\end{proof}


\section{Obstructions without twins}

 The purpose of this section is to prove the the aforementioned bound on the order of a minimal digraph $M$-obstruction without false or true twins, {\em i.e.}, the following theorem.
 
\begin{thm} \label{notwins}
Let $D$ be a minimal $M$-obstruction.   If $D$ contains neither a pair of false twins nor a pair of true twins, then $|V(D)| \le (k+1)(\ell+1)$.
\end{thm}

We will handle the proof of Theorem \ref{notwins} in two cases, each covered by one of the two following lemmas.   The cases are $k \ell=0$ and $k \ell > 0$.

\begin{lem} \label{nottwinszero}
Let $D$ be a minimal $M$-obstruction.   If $D$ does not contain a pair of false twins and $\ell = 0$, then $|V(D)| \le (k+1) = (k+1)(\ell+1)$.   Similarly, if $D$ does not contain a pair of true twins and $k = 0$, then $|V(D)| \le (\ell+1) = (k+1)(\ell+1)$.
\end{lem}

\begin{proof}
By Theorem \ref{nott} there is a vertex $v \in V(D)$ such that $D - v$ does not contain a pair of false twins.   But $D$ is a minimal $M$-obstruction, thus, $D-v$ has an $M$-partition $(V_1, \dots, V_k)$.   Since $\ell = 0$, the part $V_i$ is a homogeneous independent set of $D-v$ for every $1 \le i \le k$.   In other words, if $|V_i| > 1$ for some $1 \le i \le k$, then there is at least one pair of false twins in $D-v$, a contradiction.   Therefore, $|V_i| = 1$ and $|V(D)| = k+1$. The similar second statement is proved analogously, by applying Theorem \ref{nott} to the complement of $D$.
\end{proof}

As a consequence of Lemma \ref{nottwinszero}, if $D$ is a minimal $M$-obstruction containing neither a pair of false twins nor a pair of true twins, and $k \ell = 0$, then $|V(D)| \le (k+1)(\ell+1)$.   So, for digraphs without twins, only the case when $k \ell > 0$ remains.   Lemma \ref{notwinsnozero} will cover this case.

\begin{lem} \label{notwinsnozero}
Let $D$ be a minimal $M$-obstruction, and let $k \ell > 0$.   If  $D$ contains neither a pair of false twins nor a pair of true twins, then $|V(D)| \le (k+1)(\ell+1)$.
\end{lem}

\begin{proof}
Let $v \in V(D)$ be a vertex such that $D-v$ does not contain a pair of false twins.   Let $(V_1, \dots, V_{k+\ell})$ be an $M$-partition of $D$.   Recalling that $V_i$ is a homogenous independent set for every $1 \le i \le k$, it is clear that $|V_i| = 1$ for every $1 \le i \le k$. 

If $|V_i| \le 2$ for every $k+1 \le i \le k+\ell$, then
\begin{IEEEeqnarray*}{rCl}
|V(D)| & = & \sum_{i=1}^{k+\ell} |V_i| + 1 \\
& = &  \sum_{i=1}^{k} |V_i| + \sum_{i=k+1}^{k+\ell} |V_i| + 1 \\
& \le &  k + 2\ell + 1.
\end{IEEEeqnarray*}
Since $k > 0$, we can conclude that $k+2\ell+1 \le (k+1)(\ell+1)$.

Else, $|V_i| \ge 3$ for some $k+1 \le i \le k+\ell$.   Let $x$ be a vertex in $V_i$ and let $(X_1, \dots, X_{k+\ell})$ be an $M$-partition of $D-x$.   Then the following claims hold.

\begin{clm} \label{C1}
For every $X_j$ such that $X_j \cap (V_i \cup \{ v \}) = \varnothing$ we have $|X_j| \le 1$ .
\end{clm}

\begin{clm} \label{C2}
If $v \in X_j$ for some $1 \le j \le \ell+1$, then $|X_j| \le 1$ .
\end{clm}

\begin{clm} \label{C3}
If $z \in X_j$ for some $z \in V_i$ and some $1 \le j \le \ell+1$, then $|X_j| \le 1$.
\end{clm}

It will follow from Claims \ref{C1}, \ref{C2} and \ref{C3} that $|V(D)| \le k+\ell+1 < (k+1)(\ell+1)$.   Now we prove the claims.

\begin{proof}[Proof of Claim \ref{C1}]
Let us assume without loss of generality that $x, y \in V_{k+1}$.   Suppose that, for some $1 \le i \le k+\ell$, $X_i \cap (V_{k+1} \cup \{ v \}) = \varnothing$ and $|X_i| \ge 2$.   Thus, $X_i$ is either an homogeneous independent set or an homogeneous strong clique in $D-x$.   But this would mean that $y$ does not distinguish the vertices in $X_i$ and, since $x$ and $y$ are true twins in $D-v$, $x$ would not distinguish the vertices in $X_i$ either.   Hence $X_i$ would contain either a pair of false twins or a pair of true twins of $D$, a contradiction.
\renewcommand{\qedsymbol}{$\blacksquare$}
\end{proof}

\begin{proof}[Proof of Claim \ref{C2}]
Let us assume without loss of generality that $x, y, z \in V_{k+1}$ and $v \in X_j$.   Then it is clear that $v$ must distinguish all vertices in $V_{k+1}$, otherwise a pair of true twins would exist in $D$.   Hence, at most one vertex in $V_{k+1}$ is non-adjacent to and from $v$; let us assume without loss of generality that $v$ is adjacent to $y$ and $z$ in $D$.   We will consider two cases.

Since $v$ is adjacent to $y$, if $y \in X_j$, then $\{v,y\}$ is a strong clique.   But $\{y, z\}$ is also a strong clique.   Since $v$ and $y$ are false twins in $D-x$, we can conclude that $\{v, y, z\}$ is a strong clique in $D$, contradicting the fact that $v$ distinguishes $y$ and $z$.   A similar argument can be followed if another member of $V_{k+1}$ belong to $X_j$.

If $V_{k+1} \cap X_j = \varnothing$, then there is a vertex $u \in V(D) \setminus (V_{k+1} \cup \{v\})$ in $X_j$.   Thus, since $x$ and $y$ are true twins in $D-v$, $u$ does not distinguish $x$ and $y$, but $v$ does; a contradiction.   
\renewcommand{\qedsymbol}{$\blacksquare$}
\end{proof}

\begin{proof}[Proof of Claim \ref{C3}]
Suppose that $y \in X_j$.   Clearly, $v \notin X_j$, so $z \notin X_j$, otherwise $v$ would not distinguish $y$ and $z$.   Hence, if $u \in X_j$, then $u \notin (A_{k+1} \cup \{v\})$.   If $1 \le j \le k$, then $u$ and $y$ are false twins in $D-x$ and  hence $z$ does not distinguish them.   Since $y$ and $z$ are true twins in $D-v$ and $\{y, u\}$ is an independent set, then $\{z, u\}$ is also an independent set.    So, $z$ distinguishes $y$ and $u$ in $D-x$, a contradiction.   If $k+1 \le j \le k+\ell$, then $u$ and $y$ are true twins in $D-x$.   But $x$ and $y$ are true twins in $D-v$, hence $u$ does not distinguish $x$ and $y$, which implies that $\{x, y, u\}$ is a strong clique in $D$.   Thus, $x$ does not distinguish $y$ and $u$ in $D$, and hence $y$ and $u$ are true twins in $D$, a contradiction.   We conclude that $X_j = \{y\}$.
\renewcommand{\qedsymbol}{}
\end{proof}
\renewcommand{\qedsymbol}{$\blacksquare \square$}
\end{proof}

Now, Theorem \ref{notwins} follows from Lemmas \ref{nottwinszero}, and \ref{notwinsnozero}.


\section{Obstructions with twins}

We begin this section establishing some auxiliary facts.

\begin{lem} \label{hc}
Let $D$ be a digraph and let $S$ be a homogeneous set of $D$.   Also, let $v \in D$ and $(V_1, \dots, V_{k+\ell})$ be an $M$-partition of $D-v$ such that $S$ intersects $V_i$.    If $S$ is a strong clique and $k+1 \le i \le k+\ell$, then $(V_1, \dots, V_k,  \dots, V_i \cup \{ v \}, \dots, V_{k+\ell})$ is an $M$-partition of $D$.   Similarly, if $S$ is independent and $1 \le i \le k$, then $(V_1, \dots, V_i \cup \{ v \}, \dots, V_{k+1}, \dots, V_{k+\ell})$ is an $M$-partition of $D$.
\end{lem}

\begin{proof}
Let $u$ be arbitrarily chosen in $S \cap V_i$ and let $V_i'$ be defined as $V_i' = (V_i \setminus \{ u \}) \cup \{ v \}$.   Since $u$ and $v$ are true twins, $(V_i, \dots, V_i', \dots, V_{k+\ell})$ is an $M$-partition of $D-u$.   Recalling that $(u,v), (v,u) \in A(D)$, from here it is easy to observe that $(V_1, \dots, V_k,  \dots, V_i \cup \{ v \}, \dots, V_{k+\ell})$ is an $M$-partition of $D$.
\end{proof}

\begin{lem} \label{boundsc}
Let $D$ be a minimal $M$-obstruction.   If $S$ is an homogeneous strong clique of $D$, then $|S| \le k+1$.   Similarly, if $S$ is a homogeneous independent set of $D$, then $|S| \le \ell+1$.
\end{lem}

\begin{proof}
Suppose that $|S| \ge k+2$.   Since $D$ is a minimal $M$-obstruction, for any $v \in S$, it must be the case that $D-v$ has an $M$-partition $(V_1, \dots, V_k, V_{k+1}, \dots, V_{k+\ell})$. Since $V_i$ is an independent set for every $1 \le i \le k$, and $V_i$ is a strong clique for every $k+1 \le i \le k+\ell$, we have $|S \cap V_i| \le 1$ for every $1 \le i \le k$.   If $\ell = 0$ this is not possible, otherwise it follows from the Pigeonhole Principle that there exists $1 \le i \le \ell$ such that $S \cap V_{k+i} \ne \varnothing$.   Therefore by Lemma \ref{hc} $(V_1, \dots, V_k,  \dots, V_{k+i} \cup \{ v \}, \dots, V_{k+\ell})$ is an $M$-partition of $D$, a contradiction.
\end{proof}

We now proceed to prove the following theorem.

\begin{thm} \label{twins}
Let $D$ be a minimal $M$-obstruction.  If $D$ contains a pair of false twins or a pair of true twins, then $|V(D)| \le (k+1)(\ell+1)$.
\end{thm}

First we handle the case when either a homogeneous strong clique or a homogeneous independent set of size at least three exists in $D$.

\begin{lem}\label{k+1sc}
Let $D$ be a minimal $M$-obstruction.    If $k \ge 2$ and $D$ contains an homogeneous strong clique $S$ with $|S| = k+1$, then $|V(D)| \le (k+1)(l+1)$.   Similarly, if $\ell \ge 2$ and $D$ contains an homogeneous independent set $S$ with $|S| = k+1$, then $|V(D)| \le (k+1)(l+1)$.
\end{lem}

\begin{proof}
Let $v \in S$ be an arbitrarily chosen vertex.   Since $D$ is a minimal $M$-obstruction, $D-v$ has an $M$-partition $(V_1, \dots, V_{k+\ell})$.   By virtue of Lemma \ref{hc}, $S \cap V_{k+i} = \varnothing$ for every $1 \le i \le \ell$.   Recalling that $V_i$ is an independent set for every $1 \le i \le k$ it is clear that for every such $i$, $|S \cap V_i| = 1$.   Hence, it follows from the fact that $S$ is a strong clique that  $M_{ij} = 1$ for every pair of distinct integers $i, j$ such that $1 \le i, j \le k$. Since $S$ is a homogeneous set, it is easy to argue that $|V_i| = 1$ for every $1 \le i \le k$.

It is clear that $V_{k+i}$ is a strong homogeneous clique of $D-v$ for every $1 \le i \le \ell$.   Since $k \ge 1$, vertex $v$ has at least one twin $u \in S$.   We can assume without loss of generality that $V_1 = \{ u \}$, and hence, $(\{v\}, V_2, \dots, V_{k+\ell})$ is an $M$-partition of $D-u$.   It follows from this observation that $V_{k+i}$ is in fact an homogeneous strong clique of $D$, and we obtain from Lemma \ref{boundsc} that $|V_{k+i}| \le k+1$ for every $1 \le i \le \ell$.   Thus, \begin{IEEEeqnarray*}{rCl}
|V(D)| & = & \sum_{i=1}^{k+\ell} |V_i| + 1 \\
& = & \sum_{i=1}^k |V_i| + \sum_{i=k+1}^{k+\ell} |V_i| + 1 \\
& \le &  k + \ell(k+1) + 1 \\
& = &  k \ell + k + \ell + 1\\
& = &  (k+1)(\ell + 1).
\end{IEEEeqnarray*}
\end{proof}

From here it is easy to prove our bound for the size of a minimal $M$-obstruction, provided that it has a sufficiently large homogeneous strong clique.

\begin{lem}\label{gen}
Let $D$ be a minimal $M$-obstruction.   If $D$ contains a homogeneous strong clique $S$ with at least three vertices, then $|V(D)| \le (k+1)(\ell +1)$.   Similarly, if $D$ contains a homogeneous independent set $S$ with at least three vertices, then $|V(D)| \le (k+1)(\ell +1)$.
\end{lem}

\begin{proof}
Let $k+1-n$ be the size of a largest homogeneous strong clique of $D$.   Proceeding by induction on $n$, it is clear that Lemma \ref{k+1sc} is the base case ($n = 0$).

Let $S$ be a homogeneous strong clique of $D$ of size $3 \le k+1-n < k+1$ and $v \in S$ an arbitrarily chosen vertex.   Since $D$ is a minimal $M$-obstruction, $D-v$ has an $M$-partition $(V_1, \dots, V_{k+\ell})$.   By virtue of Lemma \ref{hc}, $S \cap V_{k+i} = \varnothing$ for every $1 \le i \le \ell$.   Also, recalling that $V_i$ is an independent set for every $1 \le i \le k$ it is clear that for every such $i$, $|S \cap V_i| \le 1$.

Also, $2 \le |S \setminus \{ v \}| = k-n$, hence, we can assume without loss of generality that $|S \cap V_i| = 1$ for $1 \le i \le k-n$.   But $S$ is a strong clique of $D$, so this implies that $M_{ij} = 1$ for every pair of distinct integers $i, j$ such that $1 \le i, j \le k-n$.   Thus, $|V_i| = 1$ for $1 \le i \le k-n$, otherwise, there would be at least one vertex $u \in V_i \setminus S$ which is non-adjacent to one vertex in $S$ and adjacent to at least one vertex in $S$ (because $2 \le k-n$), contradicting the fact that $S$ is an homogeneous set of $D$.   Since $S$ is a homogeneous set and $S \setminus \{ v \} \ne \varnothing$, it is easy to observe that $V_i$ is a homogeneous independent set of $D$ for every $1 \le i \le k$ and a homogeneous strong clique of $D$ for every $k+1 \le i \le k+ \ell$.   Hence, using Lemma \ref{boundsc} and the induction hypothesis we obtain the following.  
\begin{IEEEeqnarray*}{rCl}
|V(D)| & = & \sum_{i=1}^{k+\ell} |V_i| + 1 \\
& = & \sum_{i=1}^{k-n} |V_i| + \sum_{i=k+1-n}^{k} |V_i| + \sum_{i=k+1}^{k+\ell} |V_i| + 1 \\
& \le &  k-n + n (\ell + 1) + \ell(k+1-n) + 1 \\
& = &  k \ell + k + \ell + 1\\
& = &  (k+1)(\ell + 1).
\end{IEEEeqnarray*}

\end{proof}

It remains to prove the same bound for every minimal $M$-obstruction such that both the size of a largest homogeneous independent set and the size of a largest homogeneous strong clique are less than or equal to $2$.

\begin{lem}\label{k=1-sc=2}
Let $D$ be a minimal $M$-obstruction.   If $k=1$ and there exists a pair of true twins in $D$, then $|V(D)| \le 2(\ell +1) = (k+1)(\ell+1)$.   Similarly, if $\ell=1$ and there exists a pair of false twins in $D$, then $|V(D)| \le 2(k +1) = (k+1)(\ell+1)$.
\end{lem}

\begin{proof}
Note that when $k=1$, by Lemma \ref{boundsc}, every homogeneous strong clique has at most two vertices.   If $\ell \le 1$, then the matrix $M$ has size less than or equal to $2$.   For symmetric matrices $M$ the conclusion follows from \cite{fedhel}; thus, up to symmetry, it remains to consider the matrix with rows $(0,0)$ and $(1,1)$.   It is not difficult to see that the minimal obstructions of this matrices have at most three vertices, cf. \cite{helher}.   Thus, we will assume that $\ell \ge 2$; hence we may apply Lemma \ref{gen} unless the size of an independent homogeneous set in $D$ is at most $2$.

Let $\{x, y\}$ be a pair of true twins in $D$, and let $(V_1, \dots, V_{\ell+1})$ be an $M$-partition of $D-x$.   Clearly $y \in V_1$.   Note that $V_i$ is a homogeneous strong clique of $D$ for every $2 \le i \le \ell+1$ and $V_1 \setminus  \{ y \}$ is a homogeneous independent set of $D$; therefore, by Lemma \ref{boundsc}, $|V_i| \le 2$ for every $2 \le i \le \ell+1$ and $|V_1| \le 3$.   We want to prove that $|V(D)| \le 2(\ell+1)$; we will proceed by contradiction.   Assume that  $|V(D)| > 2(\ell+1)$, \emph{i.e.}, $|V(D)| \ge 2\ell + 3$.

\begin{clm*}
If $w \in V_1 \setminus \{y\}$, then $w$ is non-adjacent to every vertex in each $V_i$ such that $|V_i| \ge 2$.
\end{clm*}

\begin{proof}[Proof of Claim]
Else, $w$ is adjacent to $z \in V_i$ for some $2 \le i \le \ell+1$ such that there exists $z \ne u \in V_i$.    Let $(X_1, \dots, X_{\ell+1})$ be an $M$-partition of $D-z$.   Again, $u \in X_1$, and $X_1 \setminus \{u\}$ and $X_i$ are homogeneous sets for $2 \le i \le \ell+1$.   Since $w$ is non-adjacent to $u$, then $w \in X_i \ne X_1$.   We will consider two cases, $|X_i|=1$ and $|X_i|=2$.

If $|X_i|=1$, then $X_i$ is the only singleton in the partition $(X_1, \dots, X_{\ell+1})$ and $|X_1|=3$.   So, $X_1= \{ u, a, b\}$.   Since there is at most one singleton in the partition $(V_1, \dots, V_{\ell+1})$, we will assume without loss of generality that there exists $a' \in V(D)$ such that $a$ and $a'$ are true twins in $D$.   Thus, $a' \notin X_1$ and there exist $a'' \in V(D)$ and $2 \le n \le \ell+1$ such that $X_n = \{ a', a'' \}$.   Therefore, $\{a, a', a''\}$ is a homogeneous strong clique in $D$, a contradiction. 

For the case $|X_i|=2$, suppose that $X_i = \{w, a\}$.   If $a$ has a false twin $a'$ in $D$, then $\{w, a, a'\}$ is a homogeneous strong clique in $D$, a contradiction.   Otherwise, $\{a\}$ is the only singleton in the partition $(V_1, \dots, V_\ell+1)$.   Since $|V(D)| \ge 2\ell + 3$, we have $|X_1| \ge 2$, so consider $b \in X_1 \setminus \{u\}$.   We know that there exists a false twin, $b'$, of $b$ in $D$.   If $\{b'\}$ is a singleton in $(X_1, \dots, X_{\ell+1})$, then $X_1 = \{ u, b, c\}$ for some vertex $c \in V(D)$.   But there is at most one singleton in each of the partitions $(V_1, \dots, V_{\ell+1})$ and $(X_1, \dots, X_{\ell+1})$.   Hence, $c$ has a false twin in $D$, $c'$, such that $c' \notin X_1$ and thus $c'$ has a false twin in $D$, $c''$.   But again, $\{c, c', c''\}$ is a homogeneous strong clique of $D$, a contradiction.
\renewcommand{\qedsymbol}{$\blacksquare$}
\end{proof}

We will consider two cases.   For Case 1 we will suppose $|V_i| = 1$ for some $2 \le i \le \ell+1$; Case 2 is when $|V_i| \ge 2$ for every $1 \le i \le \ell+1$.

In Case 1, since $|V(D)| \ge 2\ell + 3$, we have $V_1 = \{ y, w_1, w_2 \}$.   Also, we will assume without loss of generality that $V_2 = \{ z \}$; notice that this is the only singleton in the partition.   Let $(W_1, \dots, W_{\ell+1})$ be an $M$-partition of $D-w_1$.   Since $w_1$ and $w_2$ are false twins in $D$, Lemma \ref{hc} implies that $w_2 \in W_i$ for some $2 \le i \le \ell+1$.   We will assume without loss of generality that $w_2 \in W_2$.   But $W_2$ is a strong clique and, from the Claim, $w_2$ is non-adjacent to every vertex in $V(D) \setminus \{z\}$; hence, $|W_2| \le 2$.   Again, for every $1 \le i \le \ell+1$, it is easy to note that $W_i$ is a homogeneous set of $D$, and thus, $|W_i| \le 2$ for $i \in \{ 1, 3, 4, \dots, \ell+1 \}$.   Recalling that $|V(D)| \ge 2\ell+3$, we conclude that $|W_2| = 2$; therefore, $z \in W_2$.   Hence, there are vertices $a, a', a'' \in V(D)$ and integers $2 \le n, n' \le \ell+1$ such that $a, a' \in V_n$ and $a', a'' \in V_{n'}$
 .   Therefore $\{a, a', a''\}$ is a homogeneous strong clique in $D$, a contradiction.

For Case 2, let $w$ be a vertex such $y, w \in V_1$ and let $(W_1, \dots, W_{\ell+1})$ be an $M$-partition of $D-w$.   If $|W_1| > 1$, then there are vertices $u, v \in W_1$ which are false twins in $D$ and have true twins $u'$ and $v'$, respectively, in $D$.   But this results in a contradiction, because $u' \notin W_1$ and hence $u'$ distinguishes $u$ and $v$.   Hence, $|W_1|=1$.   Since $|V(D)| \ge 2\ell + 3$, there is an integer $2 \le i \le \ell+1$ such that $|W_i|=3$.   But such $W_i$ would be a homogeneous strong clique of $D$ with $3$ vertices, unless one of the vertices in $W_i$ is a vertex $w'$ such that $w' \in V_1 \setminus \{y,w\}$.   But following the Claim, such $w'$ would be non-adjacent to every vertex in $D$, contradicting that $W_i$ is a strong clique.   So, Case 2 is not possible.

Since in every case a contradiction arose from the assumption $|V(D)| > 2(\ell+1)$, we conclude that $|V(D)| \le 2(\ell+1)$.
\end{proof}

\begin{lem}\label{sc=2}
Let $D$ be a minimal $M$-obstruction.   If the size of a maximum homogeneous strong clique of $D$ is $2$, then $|V(D)| \le (k+1)(\ell +1)$.   Similarly, if the size of a maximum homogeneous independent set of $D$ is $2$, then $|V(D)| \le (k+1)(\ell +1)$.
\end{lem}

\begin{proof}
If $\ell=0$, then Lemma \ref{boundsc} implies that there are no false twins in $D$.  Hence, Lemma \ref{nottwinszero} gives us the desired bound.

Since there is a homogeneous set of size $2$, Lemma \ref{boundsc} implies $1 \le k$.  If $k=1$, then Lemma \ref{k=1-sc=2} implies the result.   Thus, we will assume $k \ge 2$.

Let $\{u,v\}$ be a homogeneous strong clique of $D$.   By the minimality of $D$, the digraph $D-v$ has an $M$-partition $(V_1, \dots, V_{k+\ell})$ and we can assume without loss of generality that $u \in V_1$.   It is easy to observe that $V_j$ is a homogeneous set of $D$ for every $2 \le j \le k + \ell$ and  $V_1 \setminus \{ u \}$ is a homogeneous independent set of $D$

By Lemma \ref{gen} we can assume that the size of a maximum homogenous independent set of $D$ is less than or equal to $2$; let us assume first that it is $2$.   Hence $\ell \ge 2$, otherwise Lemma \ref{k=1-sc=2} would imply the result.    Also, $|V_j| \le 2$ for every $2 \le j \le k + \ell$ and $|V_1| \le 3$.

If $|V_j| = 2$ for every $2 \le j \le k + \ell$ and $|V_1| = 3$, then there are vertices $x,y \in V(D) \setminus \{ u, v \}$ such that $V_1 = \{ u, x, y \}$.   Let $(X_1, \dots, X_{k+l})$ be an $M$-partition of $D-x$.   Since $y$ is a twin of $x$, if $y \notin X_i$, then $X_i$ is an homogeneous set of $D$ for every $1 \le i \le k + \ell$ and hence $|X_i| = 2$.   Lemma \ref{hc} implies that $y \notin X_j$ for every $1 \le j \le k$.

If $y \in X_j$ for some $k+1 \le j \le k+\ell$, then $|X_j| = 3$ and $X_j$ is a homogeneous strong clique of $D-x$.   Let $a,b \in V(D) \setminus \{ x, y \}$ be vertices such that $X_j = \{ y, a, b \}$.   Since $\{u, x, y \}$ is an independent set, $\{ a, b \} \cap \{ u, v \} = \varnothing$.    Now, $u \in X_i$ for some $i \ne j$, $1 \le i \le k + \ell$.   Recall that $\{u, y \}$ is an independent set of $D$ and an homogeneous set of $D-v$, hence $\{ u, a \}$ is a strong clique of $D$.   But this contradicts the fact that $\{ y, a, b \}$ is an homogeneous set of $D-x$.   Hence, either $|V_1| \le 2$ or $|V_i| \le 1$ for some $2 \le i \le k + \ell$.

Hence,
\begin{IEEEeqnarray*}{rCl}
|V(D)| & = & \sum_{i=1}^{k+\ell} |V_i| + 1 \\
& = &  |V_1| + \sum_{i=2}^{k} |V_i| + \sum_{i=k+1}^{k+\ell} |V_i| + 1 \\
& \le &  3 + 2 (k-1) - 1 + 2 \ell + 1 \\
& = &  2 (k + \ell) + 1.
\end{IEEEeqnarray*}

We can observe that $2 (k + \ell) + 1 > (k+1)(\ell + 1)$ if and only if $k+\ell > k \ell$, and the last inequality is satisfied if and only if either $k \le 1$ or $\ell \le 1$.   Since $k \ge 2$ and $\ell \ge 2$ the desired result is then obtained.

If $D$ has no homogeneous independent sets, then $|V_1| \le 2$,  $|V_j| \le 1$ for every $2 \le j \le k$ and $|V_j| \le 2$ for every $k+1 \le j \le \ell$.   Therefore,
\begin{IEEEeqnarray*}{rCl}
|V(D)| & = & \sum_{i=1}^{k+\ell} |V_i| + 1 \\
& = &  |V_1| + \sum_{i=2}^{k} |V_i| + \sum_{i=k+1}^{k+\ell} |V_i| + 1 \\
& \le &  2 + k-1 + 2 \ell + 1 \\
& = &  k + 2 \ell + 2.
\end{IEEEeqnarray*}

We can observe that $k + 2 \ell + 2 > (k+1)(\ell + 1)$ if and only if $\ell+1 > k \ell$, and the last inequality is satisfied if and only if $k \le 1$ or $\ell = 0$.   But $k \ge 2$ and we are assuming that $\ell > 0$, thus the result follows.
\end{proof}

Now Theorem \ref{twins} follows from Lemmas, \ref{gen}, \ref{sc=2}. Moreover, Theorems \ref{notwins} and \ref{twins} imply our main result below.

\begin{thm}
Let $M$ be a $\{ 0, 1 \}$-pattern.   If $D$ is a minimal $M$-obstruction, then $|V(D)| \le (k+1)(\ell+1)$.
\end{thm}


\end{document}